\documentclass[graybox]{svmult}

\usepackage{mathptmx}       
\usepackage{helvet}         
\usepackage{courier}        
\usepackage{type1cm}        
\usepackage{makeidx}         
\usepackage{graphicx}        
\usepackage{multicol}        
\usepackage[bottom]{footmisc}

\usepackage{amsmath,amsfonts}

\definecolor{forestgreen}{rgb}{0.13,0.54,0.13}

\begin{document}

\title*{The dihedral Dunkl--Dirac symmetry algebra with negative Clifford signature }

\author{Alexis Langlois-R\'emillard } 

\institute{Alexis Langlois-R\'emillard \at Department of Applied Mathematics, Computer Science and Statistics, Faculty of Sciences, Ghent University, Krijgslaan 281-S9, 9000 Gent, Belgium \newline \email{Alexis.LangloisRemillard@UGent.be}}
\maketitle

\abstract{The Dunkl--Dirac symmetry algebra is an associative subalgebra of the tensor product of a Clifford algebra and the faithful polynomial representation of a rational Cherednik algebra. In previous work, the finite-dimensional representations of the Dunkl--Dirac symmetry algebra in three dimensions linked with a dihedral group were given. We give here the necessary results to proceed to the same construction when the Clifford algebra in the tensor product has negative signature.}

\section{Introduction}
\label{sec:1}

Dunkl operators~\cite{Dunkl} generalise partial derivatives by introducing terms related to a reflection group $W \subset \mathcal{O}(N)$, its associated root system $R$, and a function $\kappa  :R \to \mathbb{C}$ invariant on the $W$-orbits. Together with the multiplication operators and the group algebra $\mathbb{C}[W]$, they generate an associative algebra $\mathcal{A}_{\kappa}$ that is the faithful polynomial representation of a rational Cherednik algebra~\cite{Etingof}.  Given a Clifford algebra $Cl(N)$, there is an $\mathfrak{osp}(1|2)$-realisation inside the tensor product $\mathcal{A}_{\kappa}\otimes Cl(N)$ generated by the Dunkl--Dirac operator obtained by changing the partial derivatives by Dunkl operators and its dual symbol. The symmetry algebra $\mathfrak{SA}$ linked to a family of $\mathfrak{osp}(1|2)$-realisations containing the Dunkl realisation mentioned was characterised abstractly in~\cite{Oste}, and it was shown in~\cite{Oste4} that it is the full $\mathfrak{osp}(1|2)$-supercentraliser. The representation theory of these algebras is only known for a few specific cases.

In a recent article~\cite{DBLROVdJ}, we constructed the finite-dimensional representations of the dihedral Dunkl--Dirac symmetry algebra $\mathfrak{SA}_{m}\subset \mathcal{A}_{\kappa}\otimes Cl(3)$, that is, the symmetry algebra of the $\mathfrak{osp}(1|2)$-realisation linked to the group $W = \mathbb{Z}_2\times D_{2m}$ acting on the three-dimensional Euclidean space. A pair of ladder operators behaving nicely under the action of the double covering $\widetilde{W}$ of the group $W$ was instrumental to this. As the construction was rather involved, only the case when the Clifford algebra had positive signature was considered, that is when the generators $e_1$, $e_2$, $e_3$ square to $1$. The goal of this short contribution is to give the needed results to proceed to the same construction in the case when the Clifford algebra has negative signature, that is $e_1$, $e_2$, $e_3$ square to $-1$. To help compare, the sign introduced is given as $\varepsilon \in \{-1,+1\}$. We study thus here the algebra $\mathfrak{SA}_{m}^{\varepsilon} \subset\mathcal{A}_{\kappa}\otimes Cl^{\varepsilon}(3)$. The complete classification of the finite-dimensional representations is long and would greatly exceed the allowed space, we refer the readers to~\cite{DBLROVdJ} for its details. We believe this contribution could be of help for interested readers who want to translate our results, since both conventions for Clifford algebras coexist and the two lead to non-isomorphic real Clifford algebras;  multiplication of the generators by $i$ gives the correspondence for complex Clifford algebras.

 In Section~\ref{sec:2} we present the general result needed for the construction. Proposition~\ref{prop:commrel} gives the commutation relations respected by the algebra, where a small sign change appears. Proposition~\ref{prop:O123sq} compares the Casimir of the $\mathfrak{osp}(1|2)$ superalgebra with central elements of $\mathfrak{SA}_{m}^{\varepsilon}$, and two signs appear. As a consequence, the factorisation of the ladder operators changes slightly as shown in Proposition~\ref{prop:lad}. The remaining steps of the construction of the finite-dimensional representations are then presented in Section~\ref{sec:3}.

\section{The dihedral Dunkl--Dirac symmetry algebra}
\label{sec:2}
In this section we present the necessary definitions and results on the dihedral Dunkl--Dirac symmetry algebra. We refer the readers to~\cite[Sec.~2 and 3]{DBLROVdJ} for more details, bearing in mind that $\varepsilon =+1$ there.

We consider the Euclidean space $\mathbb{R}^3$ with coordinate vectors $\xi_1,\xi_2,\xi_3$ and its canonical bilinear form $\langle -,-\rangle$. Let $W= \mathbb{Z}_2 \times D_{2m}$. Its root system $R$ is  
\begin{equation}\label{eq:root}
	R= \{ \alpha_0:=(0,0,1), -\alpha_0, \alpha_j:=(\sin (j\pi/m), - \cos(j\pi/m), 0) \mid 1\leq j \leq 2m\}. 
\end{equation}
The positive root system is $R_+ = \{ \alpha_0, \alpha_1,\dots,\alpha_m\}$ and the simple roots are given by $\alpha_0$, $\alpha_1$ and $\alpha_m$. The related reflections $\sigma_{\alpha}(x) := x- 2\langle x,\alpha\rangle/\langle \alpha, \alpha\rangle$ are given in matrix form by
\begin{equation}\label{eq:refl}
	\sigma_0 := \sigma_{\alpha_0} = \begin{pmatrix}
	            	1&0&0\\
	            	0&1&0\\
	            	0&0&-1
	            \end{pmatrix},
            \quad
    \sigma_{j} := \sigma_{\alpha_j}= \begin{pmatrix}
                  	\cos(2j\pi/m)&\sin(2j\pi/m)&0\\
                  	-\sin(2j\pi/m)&\cos(2j\pi/m)&0\\
                  	0&0&1
                  \end{pmatrix}.
\end{equation}
Let $\kappa: R \to \mathbb{C}$ be a function invariant on the $W$-orbits. The Dunkl operators are
\begin{equation}
	\mathcal{D}_jf(x) := \partial_{x_j}f(x) + \sum_{\alpha\in R^+} \kappa(\alpha)\frac{f(x) -f(\sigma_{\alpha} (x))}{\left\langle x,\alpha\right\rangle} \left\langle \alpha,\xi_j\right\rangle.
\end{equation}
The Dunkl operators, the group algebra $\mathbb{C}[W]$ and the multiplication operators generate a faithful representation denoted $\mathcal{A}_{\kappa}$ of a rational Cherednik algebra.

Let $\varepsilon\in \{-1,+1\}$ be a sign and $Cl^{\varepsilon}(3)$ be the Clifford algebra generated by the three anticommuting elements $e_1$, $e_2$, $e_3$ subject to
\begin{equation}
	\{e_j,e_k\} = 2\varepsilon\delta_{ij}.
\end{equation}

There is an $\mathfrak{osp}(1|2)$-realisation given by the Dunkl--Dirac operator $\underline{\mathcal{D}}$ and its dual symbol $\underline{x}$ in the tensor product $\mathcal{A}_{\kappa}\otimes Cl^{\varepsilon}(3)$:
\begin{equation}
\underline{\mathcal{D}} = \mathcal{D}_1e_1 + \mathcal{D}_2e_2 + \mathcal{D}_3e_3, \qquad \underline{x} = x_1e_1 + x_2e_2 + x_3e_3.	
\end{equation}

We are interested in the elements of $\mathcal{A}_{\kappa}\otimes Cl(3)$ supercommuting with the $\mathfrak{osp}(1|2)$-realisation,  obtained in previous work~\cite{Oste}. First, the following elements in $W\otimes Cl^{\varepsilon}(3)$ anticommute with $\underline{\mathcal{D}}$ and $\underline{x}$:
\begin{equation}\label{eq:dcover}
	\widetilde{\sigma}_{\alpha} = \sigma_{\alpha} \otimes \sum_{j=1}^3 \langle \alpha, \xi_j\rangle e_j.
\end{equation}
They generate a group that is isomorphic to either one of the two possible central extensions of $W$ denoted by $\widetilde{W}^{\varepsilon}$~\cite{Morris}. The simple roots become $\widetilde{\sigma}_0:= \widetilde{\sigma}_{\alpha_0}$, $\widetilde{\sigma}_1 := \widetilde{\sigma}_{\alpha_1}$, $\widetilde{\sigma}_m := \widetilde{\sigma}_{\alpha_1}$ and they respect the following relations depending on the parity of $m$ and the value of $\varepsilon$:
\begin{equation}
	\widetilde{\sigma}_j^2 = \varepsilon,\quad (\widetilde{\sigma}_0\widetilde{\sigma}_m)^2 = -1,\quad (\widetilde{\sigma}_1\widetilde{\sigma}_m)^m = (-1)^{m+1}\varepsilon^m.
\end{equation}

The following linear combinations, called \emph{one-index symmetries}, of $\mathbb{C}[W]\otimes Cl^{\varepsilon}(3)$ are distinguished:
\begin{equation}\label{eq:Oj}
	O_j = \sum_{k=0}^m \kappa(\alpha_k) \langle \alpha_k,\xi_j\rangle \widetilde{\sigma}_{\alpha_{k}} = \frac{\varepsilon}{2} ([\underline{\mathcal{D}},x_i] - e_i) = \frac{\varepsilon}{2} \left( \sum_{k=1} e_k [\mathcal{D}_k,x_j] - e_j\right).
\end{equation}
Defining $L_{ij} := x_i\mathcal{D}_j - x_j\mathcal{D}_i$, the following elements, named the \emph{two-index symmetries}, commute with $\underline{\mathcal{D}}$ and $\underline{x}$
\begin{align}
	O_{ij} &:= L_{ij} + \frac{\varepsilon}{2} e_ie_j + O_ie_j - O_je_i ,\label{eq:Oijl}\\
	&= L_{ij} + \frac{\varepsilon}{2} e_ie_j + e_iO_j - e_jO_i.\label{eq:Oijr}
\end{align}
The final symmetry is named \emph{three-index symmetry} and is given by
\begin{align}
	O_{123} &= -\frac{\varepsilon}{2}e_1e_2e_3 - O_1e_2e_3 - O_2e_3e_1 - O_3 e_1e_2 + O_{12}e_3 + O_{31}e_2 + O_{23}e_1, \label{eq:O123l}\\
	 &= -\frac{\varepsilon}{2}e_1e_2e_3 - e_2e_3O_1 - e_3e_1O_2 - e_1e_2O_3 + e_3O_{12} + e_2O_{31} + e_1O_{23}. \label{eq:O123r}
\end{align}

\begin{definition}
	The dihedral \emph{Dunkl--Dirac symmetry algebra} $\mathfrak{SA}_{m}^{\varepsilon}$ is the  associative subalgebra of $\mathcal{A}_{\kappa}\otimes Cl^{\varepsilon}(3)$ generated by $O_{12}$, $O_{31}$, $O_{23}$, $O_{123}$ and the group algebra $\mathbb{C}[\widetilde{W}^{\varepsilon}]$.
\end{definition}
It is the full centraliser of the $\mathfrak{osp}(1|2)$-realisation~\cite{Oste4}.

\begin{proposition}\label{prop:commrel}
	The element $O_{123}$ commutes with every element of $\mathfrak{SA}_m^{\varepsilon}$; the two-index symmetries respect
	\begin{equation}\label{eq:Oijcomm}
		\begin{aligned}
			{[O_{12},O_{31}]} &= O_{23} + 2O_1O_{123} + \varepsilon [O_2,O_3],\\
			[O_{23},O_{12}] &= O_{31} + 2 O_{2}O_{123} + \varepsilon[O_{3},O_{1}],\\
			[O_{31},O_{23}] &= O_{12} + 2 O_3O_{123} + \varepsilon [O_1,O_2],
		\end{aligned}
	\end{equation}
	and the elements of $\widetilde{W}^{\varepsilon}$ interact as
	\begin{equation}\label{eq:commWO}
		\begin{aligned}
			\widetilde{\sigma}_0 O_{12} &= O_{12} \widetilde{\sigma}_0, &	\widetilde{\sigma}_j O_{12} &= - O_{12}\widetilde{\sigma}_j, \\
			\widetilde{\sigma}_0 O_{31} &= -O_{31} \widetilde{\sigma}_0,  &\widetilde{\sigma}_j O_{31} &= (\cos(2j\pi/m) O_{31} + \sin (2j\pi /m) O_{23})\widetilde{\sigma}_j,  \\ 
			\widetilde{\sigma}_0 O_{23} &= -O_{23} \widetilde{\sigma}_0, & \widetilde{\sigma}_j O_{23} &= (- \cos(2j\pi/m)O_{31} + \sin(2j\pi/m) O_{23})\widetilde{\sigma}_j.
		\end{aligned}
	\end{equation}
\end{proposition}
\begin{proof}
The relations~\eqref{eq:Oijcomm} come from~\cite[Thm~3.12]{Oste}.  For~\eqref{eq:commWO}, remark that it is equivalent to consider $\sigma_k L_{ij}$ by the definition~\eqref{eq:Oijl} of $O_{ij}$ and that~\eqref{eq:dcover} of $\widetilde{\sigma_j}$. Then only $\widetilde{\sigma}_jO_{12}$ is not direct, and we get:
\begin{align*}
\sigma_j L_{12} &= \sigma_j (x_1\mathcal{D}_2 - x_2\mathcal{D}_1) = (\sin^2(2j\pi /m) + \cos^2(2j\pi/m)) L_{21} \sigma_j = - L_{12}\sigma_j. 
\end{align*}
Working out the remaining terms of \eqref{eq:Oijl} gives the rest. \hfill \qed
\end{proof}

We are interested in the representation theory of $\mathfrak{SA}_m^{\varepsilon}$. The construction uses ladder operators, and their factorisations in turn follow from  the next proposition.

\begin{proposition}\label{prop:O123sq}
	The three-index symmetry squares to
	\begin{equation}
		O_{123}^2 = -\frac{\varepsilon}{4} + O_1^2 + O_2^2 + O_3^2 + \varepsilon(O_{12}^2 + O_{31}^2 + O_{23}^2).
	\end{equation}
\end{proposition}
\begin{proof}
	Express $O_{123}^2$ as the product of the two expressions~\eqref{eq:O123l} and~\eqref{eq:O123r} 
\begin{align}
O_{123}^2 &=  (-\frac{\varepsilon}{2}e_1e_2e_3 - O_1e_2e_3 - \underline{O_2e_3e_1} - O_3 e_1e_2 + O_{12}e_3 + O_{31}e_2 + O_{23}e_1 ) \nonumber \\
	 &\quad \times ( -\frac{\varepsilon}{2}e_1e_2e_3 - e_2e_3O_1 - e_3e_1O_2 - e_1e_2O_3 + \underline{e_3O_{12}} + e_2O_{31} + e_1O_{23}\nonumber\\
&= -\frac{\varepsilon}{4} - O_1^2 - O_2^2 - O_3^2 + \varepsilon(O_{12}^2 + O_{31}^2 + O_{23}^2) + Q,
\end{align}
where $Q$ expresses the 42 remaining ``cross terms''. We show now that $Q = 2(O_1^2+O_2^2+O_3^2)$. Replace in $Q$ all instances of $O_{ij}$ on the left with~\eqref{eq:Oijl}, and all instances on the right by~\eqref{eq:Oijr}. For example, the terms below produce $2(O_1^2+O_2^2+O_3^2)$ (the underlined term comes from the two underlined terms in the product)
\begin{align*}
        A&= \varepsilon( (\underline{O_{2} e_1 O_{12}} - O_{12} e_1O_{2}) + (O_{1}e_3 O_{31} - O_{31}e_3O_{1}) + (O_{3}e_2O_{23} - O_{23}e_2O_3))\\
        &= \varepsilon \Big( \underline{O_2e_1L_{12} + \tfrac{1}{2} O_2 e_2 + \varepsilon O_2^2 - O_2 e_1e_2O_1} - L_{12} e_1 O_2+ \tfrac12 e_2 O_2 + O_1e_1e_2O_2 + \varepsilon O_2^2\\
        &\quad + O_1 e_3 L_{31} + \tfrac12 O_1 e_1 + \varepsilon O_1^2 + O_1 e_1e_3 O_3-  L_{31} e_3 O_1 + \tfrac12 e_1 O_1 + O_3 e_3e_1 O_1 + \varepsilon O_1^2\\
        &\quad + O_3 e_2 L_{23} +\tfrac12 O_2 e_3 + \varepsilon O_3^2 + O_3 e_3e_2 O_2 - L_{23}e_3O_3 + \tfrac12 e_3O_3 + O_2 e_2e_3O_3 + \varepsilon O_3^2\Big)\\
        &= 2\varepsilon^2(O_1^2+O_2^2+O_3^2) + B, \qquad \text{with $B$ the remaining part}.
\end{align*}
After doing this procedure for all terms, and further simplifications, one reaches
\begin{align*}
	Q&= 2(O_1^2+O_2^2 +O_3^2) \\
	&+ \frac{\varepsilon}{2}\left(\begin{smallmatrix}
	   	\ L_{12}e_1e_2(\varepsilon-e_3e_1 L_{31} - e_2e_3 L_{23}+2\varepsilon e_3O_3)\\
	   	+L_{31}e_3e_1(\varepsilon -e_1e_2L_{12} -e_2e_3L_{23} +2\varepsilon e_2O_2)\\
	   	+L_{23}e_2e_3(\varepsilon-e_1e_2L_{12} - e_3e_1L_{31} +2\varepsilon e_1O_1)
	   \end{smallmatrix} \right)
	   +\frac{\varepsilon}{2}
	   \left(\begin{smallmatrix}
	   	\ (\varepsilon-L_{31}e_3e_1 - L_{23}e_2e_3 + 2\varepsilon O_3e_3)e_1e_2L_{12}\\
	   	+ (\varepsilon-L_{12}e_1e_2 - L_{23}e_2e_3 + 2\varepsilon O_2e_2)e_3e_1L_{31}\\
	   	+ (\varepsilon-L_{12}e_1e_2 - L_{31}e_3e_1 + 2\varepsilon O_1e_1)e_2e_3L_{23}
	   \end{smallmatrix}\right).
\end{align*}
The last line is zero. To prove this, replace the $O_j$ by their last definition~\eqref{eq:Oj} in terms of commutators $C_{kj}:=[\mathcal{D}_k,x_j]$ and apply the following identity~\cite[Thm~2.5]{Oste}
\begin{equation}
	L_{ij}L_{kl} + L_{ki}L_{jl} + L_{jk}L_{il} = L_{ij}C_{kl} + L_{ki}C_{jl} + L_{jk}C_{il},
\end{equation}
keeping in mind that $L_{ii} = 0$, $L_{ij}= -L_{ji}$ and $C_{ij} = C_{ji}$. \hfill \qed
\end{proof}

This proposition yields in fact a correspondence between the Casimir of the Lie algebra $\mathfrak{osp}(1|2)$ and a central element in the symmetry algebra. Similar statements hold for any reflection group in any dimension, see~\cite{Oste4}.

The finite-dimensional representations are constructed via ladder operators. In the classical non-Dunkl case, the ladder operators for the $\mathfrak{so}(3)$ algebra are given by the following linear combinations of the two-index symmetries:
\begin{equation}\label{eq:O0OpOm}
O_0 := -i O_{12}, \quad O_+ := iO_{31} + O_{23}, \quad O_{-} := iO_{31} - O_{23}.
\end{equation}
For ease of notation, denote the following combination of one-index symmetries (note that they vanish when $\kappa =0$):
\begin{equation}\label{eq:T0TpTm}
T_0 := iO_3, \quad T_+ := O_1 + iO_2, \quad T_- := O_1- iO_2.
\end{equation}

\begin{proposition}\label{prop:commrelOpOm}
The commutation relations respected by $O_0$, $ O_+$ and $O_-$ are
\begin{equation}
\begin{aligned}
{[O_0,O_+]} &= +O_{+} + \{ O_{123}, T_+\} + \varepsilon [T_0,T_+],\\
[O_0,O_-] &= -O_{-} + \{ O_{123}, T_-\} - \varepsilon [T_0,T_-],\\
[O_0,O_+] &= \ 2O_0 - \{ O_{123}, T_0\} + \varepsilon [T_+,T_-],
\end{aligned}
\end{equation}
and those with $T_0$, $T_+$ and $T_-$ are
\begin{equation}
	\begin{aligned}
    T_0O_0 &=  O_0T_0, & T_0O_+ &= - O_+T_0, & T_0 O_- &= - O_-T_0, \\
    T_+ O_0 &= -O_0T_+, & T_+O_- &= -O_+T_-, & T_-O_+ &= O_-T_+, \\
    T_- O_0 &= -O_0 T_-, & T_- T_0&= - T_0T_-, & T_+ T_0&= -T_0 T_+. 
	\end{aligned}
\end{equation}

\end{proposition}
\begin{proof}
Use the commutation relations of Proposition~\ref{prop:commrel}. \hfil \qed
\end{proof}

In this new basis, the following expressions hold.
\begin{proposition}\label{prop:factoOpOm}
The square of the three-index symmetry becomes
\begin{align}
O_{123}^2 &= -\frac{\varepsilon}{4} + T_+T_- - T_0^2 - \varepsilon( O_0^2 - O_0 + O_+O_- + 2O_{123}T_0),\label{eq:sq:O123inOpOm}\\
&= -\frac{\varepsilon}{4} + T_-T_+ - T_0^2 - \varepsilon( O_0^2 + O_0 - O_-O_+ - 2O_{123}T_0).\label{eq:sq:O123inOmOp}
\end{align}
Furthermore, the following equations hold
\begin{align}
O_+O_- &=  \varepsilon T_+T_- - (O_0 - 1/2)^2 - \varepsilon(\varepsilon O_{123} + T_0)^2, \label{eq:facOpOm}\\
O_-O_+ &=  \varepsilon T_-T_+ - (O_0 + 1/2)^2 - \varepsilon(\varepsilon O_{123} - T_0)^2\label{eq:facOmOp}.
\end{align}
\end{proposition}
\begin{proof}
We prove~\eqref{eq:sq:O123inOpOm} by directly rewriting from the definitions~\eqref{eq:O0OpOm} and~\eqref{eq:T0TpTm}:
\begin{align*}
O_{12}^2 &=  -O_0^2, & O_{31}^2 + O_{23}^2 &= - O_+O_- + O_0 - 2O_{123}T_0 + \frac{\varepsilon}{2} [T_+,T_-],\\
O_3^2 &=  T_0^2, & O_1^2+O_2^2 &= T_+T_-  -\frac{1}{2} [T_+,T_-].
\end{align*}
Equation~\eqref{eq:sq:O123inOmOp} is similar, and the expressions~\eqref{eq:facOpOm} and~\eqref{eq:facOmOp} follow directly.\hfill \qed
\end{proof}

\begin{proposition}\label{prop:lad}
The following operators 
\begin{equation}\label{eq:lad}
L_+ := \frac{1}{2} \{ O_0, O_+\} \qquad \text{and} \qquad L_- := \frac{1}{2} \{O_0, O_-\}
\end{equation}
are ladder operators with respect to $O_0$ in the sense that
\begin{equation}
{[ O_0, L_+]} = + L_+, \qquad [O_0, L_-] = - L_-,
\end{equation}
and the products of two of them admit the following factorisations
\begin{align}
L_+L_- &= -((O_0 - 1/2)^2 + \varepsilon (\varepsilon O_{123} + T_0)^2)((O_0 - 1/2)^2 - \varepsilon T_+T_-),\label{eq:LpLm}\\
L_-L_+ &= -((O_0 + 1/2)^2 + \varepsilon (\varepsilon O_{123} - T_0)^2)((O_0 + 1/2)^2 - \varepsilon T_-T_+).\label{eq:LmLp}
\end{align}
\end{proposition}
\begin{proof}
That $L_+$ and $L_-$ are ladder operators comes from Proposition~\ref{prop:commrelOpOm}
\begin{align*}
{2[O_0, L_{\pm}]} &= [O_0, \{ O_0, O_{\pm}\}] = \{ O_0, [O_0, O_{\pm}]\} \\
&= \{O_0, \pm O_{\pm} + \{O_{123},T_\pm\} \pm \varepsilon[T_0,T_{\pm}]\} = \pm \{O_0,O_{\pm}\} = 2L_{\pm} ,
\end{align*}
where equation~\eqref{eq:commWO} was used in the second line. The proof of the factorisation is the same as~\cite[Prop~3.8]{DBLROVdJ} using the $\varepsilon$ variants of the commutation relations.\hfill \qed
\end{proof}

\section{Sketch of the finite-dimensional representations construction}
\label{sec:3}

Everything needed for the construction of the finite-dimensional representations is in place. Doing it would, however, greatly exceed the scope of this note. We give below a sketch of the steps needed and refer the readers to~\cite{DBLROVdJ} for the details. 

\begin{enumerate}
\item Any finite-dimensional $\mathfrak{SA}_{m}^{\varepsilon}$-representation decomposes as a $\widetilde{W}^{\varepsilon}$-representation into a direct sum of spin irreducible $\widetilde{W}^{\varepsilon}$-representations by Maschke's Theorem (the irreducible representations for these groups can be found in~\cite[Thm~A.5]{DBLROVdJ}). Let $\widetilde{W}_0^{\varepsilon}$ be the subgroup of $\widetilde{W}^{\varepsilon}$ generated by elements commuting with $O_0$. The associative subalgebra of $\mathfrak{SA}_{m}^{\varepsilon}$ generated by  $O_0$, $L_+$, $L_-$, $O_{123}$ and $\widetilde{W}^{\varepsilon}_0$ has a triangular decomposition. Use this triangular decomposition and the ladder operators to give a basis of $O_0$- and $O_{123}$-eigenvectors for any irreducible $\mathfrak{SA}_{m}^{\varepsilon}ˇ$-representation. (See~\cite[Lem.~4.3]{DBLROVdJ}.)
\item Thus start from a general $O_0$- and $O_{123}$-eigenbasis. The elements $v_j^+$ and $v_j^-$ of this basis are obtained from multiple applications of the ladder operators on a first pair $v_0^+,v_0^-$. Use the two factorisations~\eqref{eq:LpLm} and~\eqref{eq:LmLp} to create equations $L_+v_j^- = A(j)v_{j+1}^-$ and $L_-v_j^+= A(j)v_{j+1}^+$.  The terms $A(j)$ will depend on the first $\widetilde{W}^{\varepsilon}$-representation, and on the eigenvalues of $O_{123}$ and $O_0$. Then irreducibility and the finite-dimension give conditions on $A(j)$. (See~\cite[(4.21)--(4.23)]{DBLROVdJ}.)
\item Solve the system obtained for the values of the $O_{123}$- and $O_{0}$-eigenvalues keeping track of the conditions on $\kappa$. (See~\cite[(4.28)]{DBLROVdJ}.)
\item Furthermore, the unitarity of the representations can be studied in the same fashion by looking at positivity constraints in the $A(j)$. (See~\cite[Sec.~3.3 and Lem.~4.4]{DBLROVdJ}.)
\end{enumerate}

%
%
%
%
%

\begin{acknowledgement}
We wish to thank Hendrik De Bie, Roy Oste and Joris Van der Jeugt for helpful discussions and support.
ALR holds scholarships from the Fonds de recherche Nature et technologies (Qu\'ebec) 270527 and the Excellence of Science Research Project 30889451. This support is gratefully acknowledged. 
\end{acknowledgement}

\end{document}